\documentclass{amsart}
\usepackage{amssymb,amsmath,latexsym}
\usepackage{amsthm}
\usepackage{fontenc}
\usepackage{amssymb}

\numberwithin{equation}{section}

\newtheorem{theorem}{Theorem}[section]
\newtheorem{corollary}{Corollary}[theorem]

\begin{document}
\author{Alexander E Patkowski}
\title{On criteria related to the reciprocal of the Riemann zeta function}

\maketitle
\begin{abstract}We explore Fourier transforms of the reciprocal of the Riemann zeta function that have connections to the RH. A partial answer to a recently posed problem is explored by exploiting the fact that $\zeta(s)\neq0$ when $\Re(s)=1.$ \end{abstract}

\keywords{\it Keywords: \rm Riesz criteria; Riemann zeta function; M$\ddot{o}$bius function}

\subjclass{ \it 2010 Mathematics Subject Classification 11L20, 11M06.}

\section{Introduction and Main Idea} 
A celebrated criteria for the truth of the Riemann Hypothesis by Hardy and Littlewood [3] is the estimate
\begin{equation}\sum_{n\ge1}\frac{(-x)^n}{n!\zeta(2n+1)}=O(x^{-1/4+\epsilon}),\end{equation}
for every $\epsilon>0.$ Titchmarsh has discussed the background of this result in [8, pg.382] (see also [7] for some commentary). In a recent paper by Dixit [1, Section 7, Number 2], it was posed as a problem to offer a Fourier transform representation of the left side of (1.1) involving the Riemann xi function $\Xi(\frac{t}{2}).$ Here $\Xi(t):=\xi(\frac{1}{2}+it),$ and [5] $\xi(s):=\frac{1}{2}s(s-1)\pi^{-\frac{s}{2}}\Gamma(\frac{s}{2})\zeta(s).$ A reasonable guess would be the integral
$$\int_{0}^{\infty}\frac{\Xi(t)\cos(\log(x)t)}{(t^2+\frac{1}{4})|\zeta(\frac{1}{2}+it)|^2}dt,$$ 
but unfortunately, it does not converge because there are infinitely many $t\in\mathbb{R}$ such that $\zeta(\frac{1}{2}+it)=0.$ As a result of our interest in this problem we were, however, led to some interesting results we believe are just as relevant. In particular we instead decided to work with the integral 
$$\int_{0}^{\infty}\frac{f(t)\cos(\log(x)t)}{|\zeta(1+2it)|^2}dt,$$
where the $f(t)$ is appropriately chosen so that the integral may be transformed into a inverse Mellin transform by using symmetry properties [8, pg.35]. This was a natural choice since, as noted in [8, pg.45], Hadamard and de la Vall$\acute{e}$e Poussin established that $\zeta(s)\neq0$ if $\Re(s)=1.$ 
\begin{theorem} Let $\mu(n)$ denote the M$\ddot{o}$bius function. Assuming the Riemann Hypothesis, we have that 
\begin{equation}\int_{0}^{\infty}\frac{\cos(xt)}{\cosh(\pi t)|\zeta(1+2it)|^2}dt=\pi\sum_{n,m\ge1}\frac{\mu(n)\mu(m)}{n^2e^{-x/2}+m^2e^{x/2}}.\end{equation}
Furthermore, we have \begin{equation}h_2(x):=\sum_{n,m\ge1}\frac{\mu(n)\mu(m)}{n^2x+m^2}=O(x^{-1/4}),\end{equation} as $x\rightarrow 0^{+},$
and
\begin{equation}h_2(x)=O(x^{-3/4+\delta'}),\end{equation} for every $\delta'>0$ as $x\rightarrow \infty.$
\end{theorem}
\begin{proof} First, we transform the left side of (1.2),
\begin{equation}\int_{0}^{\infty}\frac{\cos(xt)}{\cosh(\pi t)|\zeta(1+2it)|^2}dt=\frac{1}{2ie^{x/2}}\int_{\frac{1}{2}-i\infty}^{\frac{1}{2}+i\infty}\frac{e^{sx}\Gamma(s)\Gamma(1-s)ds}{\zeta(2s)\zeta(2-2s)}.\end{equation} For the left hand side of (1.5) note that since [4, pg.104] $$\frac{1}{\zeta(1+it)}=O(\log(t)),$$ we have 
$$\int_{0}^{\infty}\frac{\cos(xt)}{\cosh(\pi t)|\zeta(1+2it)|^2}dt=O\left(\int_{0}^{\infty}e^{-\pi t}\cos(xt)\log(2|t|)^2dt\right),$$ which shows convergence. In the case of the right hand side of (1.5), note that for $\Re(s)>\frac{1}{2},$ $x>0,$ we have 
\begin{equation}\int_{0}^{\infty}t^{s-1}\left(\sum_{n\ge1}\mu(n)e^{-n^2xt}\right)dt=\frac{\Gamma(s)}{x^s\zeta(2s)}.\end{equation} If we now assume the Riemann hypothesis [8, Theorem 14.25(A)] we may extend the region of holomorphy to $\Re(s)>\frac{1}{4}.$ From Parseval's formula for Mellin transforms [6, pg.82] with $h(t)=\sum_{n\ge1}\mu(n)e^{-n^2t}$ and (1.6) we have that
$$\int_{0}^{\infty}h(xt)h(t)dt=\frac{1}{2\pi i}\int_{a-i\infty}^{a+i\infty} \frac{\Gamma(s)\Gamma(1-s)}{x^s\zeta(2s)\zeta(2-2s)}ds=\sum_{n,m\ge1}\frac{\mu(n)\mu(m)}{n^2x+m^2},$$ for $\frac{1}{4}<a<\frac{3}{4},$ and $x>0.$ This may also be computed using known Mellin transforms, such as [6, pg.88, eq.(3.2.6)]. On the other hand, using the concept from [8, pg.35], we can transform the integral to
$$\int_{0}^{\infty}\frac{\cos(xt)}{\cosh(\pi t)|\zeta(1+2it)|^2}dt=\frac{e^{-x/2}}{2i}\int_{\frac{1}{2}-i\infty}^{\frac{1}{2}+i\infty} e^{xs}\frac{\Gamma(s)\Gamma(1-s)}{\zeta(2s)\zeta(2-2s)}ds,$$
which proves (1.2).

Recall that we defined $h_2(x)$ to be the absolutely convergent series for $x>0,$
$$h_2(x):=\sum_{n,m\ge1}\frac{\mu(n)\mu(m)}{n^2x+m^2}.$$ Moving the line of integration to the region $-1<b<0$ gives (assuming $\rho=\frac{1}{2}+i\gamma$ are simple)
\begin{equation}h_2(x)=\frac{x^{-1/4}}{2}\sum_{\gamma: \zeta(\frac{1}{2}+i\gamma)=0}\frac{x^{-i\gamma/2}\Gamma(\frac{1}{4}+i\frac{1}{2}\gamma)\Gamma(\frac{3}{4}-i\frac{1}{2}\gamma)}{\zeta'(\frac{1}{2}+i\gamma)\zeta(\frac{3}{2}-i\gamma)}-\frac{12}{\pi^2}\end{equation}
$$+\frac{1}{2\pi i}\int_{b-i\infty}^{b+i\infty} \frac{\Gamma(s)\Gamma(1-s)}{x^s\zeta(2s)\zeta(2-2s)}ds.$$ If we assume the first series on the right hand side of (1.7) converges absolutely, we may similarly argue as in [3, pg.161] to find we have that it is $O(x^{-1/4}).$ As the remaining poles at the integers $s=-n,$ $n\ge1,$ of the integrand in (1.7) give rise to residues with higher order terms $O(x^N)$ and $O(x^N\log(x))$ for $N\ge1.$ To see how $O(x^N\log(x))$ arises, note that these are double poles because of $\Gamma(s)$ and $1/\zeta(2s)$ in the integrand. The asymptotic estimate (1.3) now follows for $x\rightarrow 0^{+}.$ In the case of (1.4), when $x\rightarrow\infty,$ we replace $s$ by $1-s$ in our integral leaving $x^{s-1}$ in the integrand. In this instance the remaining terms are $O(x^{-M})$ and $O(x^{-M}\log(x)),$ $M\ge1,$ after computing the residue at the non-trivial zeros. This implies that the growth as $x\rightarrow\infty$ is closer in behavior to the first computed residue at the non-trivial zeros $O(x^{-3/4}),$ as the other residues tend to $0$ faster. \end{proof}
\par We mention that our $h(t)$ may be replaced with the Riesz function $\sum_{n\ge1}\mu(n)n^{-2}e^{-t/n^2}$ [8, pg.382] in the proof to arrive at the result. This implies that our main theorem is related to Riesz criteria through a Mellin convolution transform. In light of the work [4], we thought it would be of interest to offer the following.

\begin{corollary} If $\Re(\beta)>0,$ then
\begin{equation}\int_{0}^{\infty}\frac{e^{-t^2/4\beta}}{|\zeta(1+2it)|^2}dt=\frac{1}{\sqrt{\pi\beta}}e^{-\pi^2\beta}\int_{0}^{\infty}e^{-t^2\beta}\cos(2\pi t\beta)e^{-t/2}h_2(e^{-t})dt. \end{equation}
Further, $$\int_{0}^{\infty}e^{-t/2}h_2(e^{-t})dt=-2\sum_{n,m\ge1}\frac{\mu(n)\mu(m)}{nm}\tan^{-1}(\frac{n}{m})=0,$$
is equivalent to the PNT.
\end{corollary}
\begin{proof} We work with the integral [2, pg.515, eq.3.898] valid for $y>0,$ $\Re(\beta)>0,$
\begin{equation}\int_{0}^{\infty}\cos(yt)e^{-t^2/(4\beta)}\cosh(\alpha t)dt=\sqrt{\pi\beta}e^{\alpha^2\beta-\beta y^2}\cos(2\alpha\beta y). \end{equation}
By Theorem 1.1, we have the integral in (1.8) after application of Fourier inversion. Namely, we compute
$$ \int_{0}^{\infty} \cos(x2\beta\pi)e^{-x^2\beta}\left(\int_{0}^{\infty}\frac{\cos(xt)}{\cosh(\pi t)|\zeta(1+2it)|^2}dt\right)dx$$
$$\begin{aligned} &=\int_{0}^{\infty}\frac{1}{\cosh(\pi t)|\zeta(1+2it)|^2}\left(\int_{0}^{\infty} \cos(x2\beta\pi)e^{-x^2\beta}\cos(xt)dx \right)dt\\
&=\int_{0}^{\infty}\frac{e^{-t^2/4\beta}}{|\zeta(1+2it)|^2}dt. \end{aligned}$$
Here we have applied our definition of $h_2(x),$ equation (1.9) with $\alpha=\pi,$ and Fubini's theorem to interchange the integrals.
For the next part, we recall the Prime Number Theorem is equivalent to the statement that [8] $$\sum_{n\ge1}\frac{\mu(n)}{n}=0.$$
By Theorem 1.1, we apply the Fourier inversion theorem again, and then set $t=0$ and integrate through by absolute convergence to get the desired formula
$$\sum_{n,m\ge1}\int_{0}^{\infty}\frac{\mu(n)\mu(m)}{n^2e^{t/2}+e^{-t/2}m^2}dt$$ $$\begin{aligned} &=\sum_{n,m\ge1}\frac{\mu(n)\mu(m)}{n^2}\int_{0}^{\infty}\frac{\partial}{\partial t} \left(2\tan^{-1}(e^{t/2}\frac{n}{m})\frac{n}{m}\right)dt \\
&=\pi\sum_{n\ge1}\frac{\mu(n)\mu(m)}{nm}-2\sum_{n,m\ge1}\frac{\mu(n)\mu(m)}{nm}\tan^{-1}(\frac{n}{m}).\end{aligned}$$ \end{proof}
We mention that the series should be expected to converge to $0$ under the PNT and the squeeze theorem, since $-\frac{\pi}{2}<\tan^{-1}(\frac{n}{m})<\frac{\pi}{2},$ for each natural number $n,m.$
\section{The Perron formula}
In this section give a further example involving the Mellin-Perron formula that appeals to the Merten's function $M(x)=\sum_{n\le x}\mu(n).$ Using the Mellin-Perron formula [6, pg.93] for $c>\frac{1}{2},$ $x\ge1,$
\begin{equation}M_2(x):=\sum_{n^2\le x}'\mu(n)=\frac{1}{2\pi i}\int_{c-i \infty}^{c+i \infty}\frac{x^s}{s\zeta(2s)}ds,\end{equation}
Where the prime indicates that the sum is to be halved if $x=1.$ Note that $M_2(x)=M(\sqrt{x})$ when $x>1.$ Hence, an equivalent form of the Riemann Hypothesis is that for every $\epsilon>0,$
\begin{equation}M_2(x)=O(x^{\frac{1}{4}+\epsilon}).\end{equation}
This follows clearly from (2.1) and [8, Theorem 14.25(A)].

\begin{theorem}We have, under the assumption of the Riemann Hypothesis,
\begin{equation}\int_{0}^{\infty}\frac{\cos(xt)}{(t^2+\frac{1}{4})|\zeta(1+2it)|^2}dt=\pi e^{-x/2}\int_{1}^{\infty}M_2(e^{-x}t)M_2(t)\frac{dt}{t^2}.\end{equation}
\end{theorem}
\begin{proof} If we assume the RH and extend the region of holomorphy of (2.1) to $\Re(s)>\frac{1}{4},$ we have for $\frac{1}{4}<d<\frac{3}{4},$
\begin{equation}\int_{1}^{\infty}M_2(xt)M_2(t)\frac{dt}{t^2}=\frac{1}{2\pi i}\int_{d-i \infty}^{d+i \infty}\frac{x^{-s}}{s(1-s)\zeta(2s)\zeta(2-2s)}ds.\end{equation}
On the other hand and we have,
\begin{equation}\int_{0}^{\infty}\frac{\cos(xt)}{(t^2+\frac{1}{4})|\zeta(1+2it)|^2}dt= \frac{e^{-x/2}}{2i}\int_{\frac{1}{2}-i \infty}^{\frac{1}{2}+i \infty}\frac{e^{xs}}{s(1-s)\zeta(2s)\zeta(2-2s)}ds.\end{equation}\end{proof}
Note that the validity of (2.3) depends on the RH, since if $M_2(x)=O(x^{\delta})$ where $\frac{1}{4}<\delta<\frac{1}{2},$ we have
$$\int_{1}^{\infty}M_2(t)M_2(t)\frac{dt}{t^2}=O\left(\int_{1}^{\infty}t^{2\delta-2}dt\right)=O(1).$$ If we only had $\delta\ge\frac{1}{2}$ then the integral would diverge.

\section{Conclusions}
First, we believe we have given a new route in exploring the applications of Fourier-type integrals in relation to RH criteria. Second, in light of our Theorem 2.1, it would seem to be of interest to explore possible connections with the weak Merten's conjecture [8, eq.(14.29.1)],
$$\int_{1}^{X^2}\left(\frac{M_2(t)}{\sqrt{t}}\right)^2dt=\int_{1}^{X}\left(\frac{M(t)}{t}\right)^2dt=O(\log(X)).$$

1390 Bumps River Rd. \\*
Centerville, MA
02632 \\*
USA \\*
E-mail: alexpatk@hotmail.com

\end{document}